\documentclass[11pt]{amsart}
\usepackage{amsmath,amsfonts,amssymb,amsthm}
\usepackage{latexsym,bm,graphicx}
\usepackage{mathrsfs}
\usepackage{color}
\usepackage{hyperref}

\title[A note on the positivity of a quasi-local mass]{A note on the positivity of a quasi-local mass in general dimensions}
\author{Xian-Tao Huang}
\address{School of Mathematics and Computational science\\  Sun Yat-sen University\\ Guangzhou 510275\\ E-mail address: hxiant@mail2.sysu.edu.cn}

\newtheorem{thm}{Theorem}[section]
\newtheorem{prop}[thm]{Proposition}

\theoremstyle{definition}
\theoremstyle{remark}

\newtheorem{rem}[thm]{Remark}

\numberwithin{equation}{section}

\begin{document}
%\today

\maketitle
\begin{abstract}Wang and Yau \cite{wangyau07} introduced a quasi-local mass, which is a hyperbolic background generalization of Liu-Yau's expression \cite{liuyau03} \cite{liuyau06}, and proved its positivity. In this note, we prove that the positivity of this quasi-local mass is still valid under weaker assumptions on the boundary hypersurface in general dimensions. The method we used is similar to that used by Eichmair, Miao and Wang in \cite{emw12}.

\vspace*{5pt}
\noindent {\it 2010 Mathematics Subject Classification}: 53C20, 83C99.

\vspace*{5pt}
\noindent{\it Keywords}: Quasi-local mass, Inverse curvature flow.

\end{abstract}
%\tableofcontents
%\setcounter{tocdepth}{1}
\section{Introduction}

In the celebrated work \cite{ShiTam02}, Shi and Tam made use of an idea of
Bartnik \cite{Bar86} and the positive mass theorem to prove a beautiful result on the boundary behavior of compact Riemannian manifolds with non-negative scalar curvature. In \cite{liuyau03} \cite{liuyau06}, Liu and Yau introduced a quasi-local mass in space-time, whose Riemannian version corresponds
to the earlier theorem of Shi and Tam, and proved its positivity.
Later on, Wang and Yau \cite{wangyau07} generalized Liu-Yau's quasi-local mass in the case when
the Gaussian curvature of the surface is not necessarily positive, and obtained a time-like four-vector instead of a positive quantity. In \cite{wangyau07}, even though Wang and Yau dealt with 2-dimensional closed surfaces, they also claimed similar result holds in higher dimensions. Wang and Yau's theorem is as follows:

\begin{thm}\label{origin} {\rm(Wang, Yau \cite{wangyau07})}
Suppose $(\Omega, g)$ is an $n$-dimensional compact spin manifold $(n\geq3)$ with smooth boundary $\Sigma$, which is a topological sphere.
Suppose the following conditions hold:
\begin{description}
  \item[(1)] the scalar curvature $R$ of $\Omega$ satisfies $R\geq-n(n-1)k^{2}$ and the mean curvature $H$ of $\Sigma$ with respect to the outward pointed unit normal vector is positive,
  \item[(2)] the sectional curvature $K$ of $\Sigma$ is larger than $-k^{2}$, and $\Sigma$ can be isometrically embedded into $\mathbb{H}^{n}_{-k^{2}}$ in $\mathbb{R}^{n,1}$.
\end{description}
Then there exists a future-directed
time-like vector-valued function $\mathbf{W^{0}}: \Sigma\rightarrow\mathbb{R}^{n,1}$ such that
\[\int_{\Sigma}(H_{0}-H)\mathbf{W^{0}}d\sigma\]
is a future-directed non-space-like vector. Here $d\sigma$ is the volume form of the induced metric, $H_{0}$ is the mean curvature of
the isometric embedding into $\mathbb{H}^{n}_{-k^{2}}$.
\end{thm}

\vskip 0.2cm \noindent

\begin{rem}
There is a space-time version in Wang and Yau's paper \cite{wangyau07}. However, we are only interested in the Riemannian version in this paper. Recently, Wang and Yau \cite{wangyau09_1} \cite{wangyau09_2} introduced a new definition of quasi-local mass, and there are lots of investigation on the new definition.
\end{rem}

\vskip 0.2cm \noindent

The proof of Theorem \ref{origin} consists of three parts, we give a short description in the following. We denote the image of the isometric embedding of $\Sigma$ in $\mathbb{H}^{n}_{-k^{2}}$ by $\Sigma_{0}$.
The first part of the proof is to deform $\Sigma_{0}$ in the normal direction at unit speed in order to obtain a foliation of the unbounded region $\overline{\Omega}$ of $\mathbb{H}^{n}_{-k^{2}}\setminus\Sigma_{0}$, then we obtain a new metric on $\overline{\Omega}$ by solving a prescribed scalar curvature equation, then gluing it with $(\Omega,g)$ to obtain an asymptotically hyperbolic manifold $(\overline{M},\overline{g})$. The second part need the existence of Killing spinor fields on hyperbolic spaces, and make use of them to prove a positive mass type theorem for the asymptotically hyperbolic manifold $(\overline{M},\overline{g})$. The third component is to obtain $\mathbf{W^{0}}$ by solving a backward parabolic equation with a prescribed value at infinity and to obtain a monotonicity formula.
The first and third parts of the argument in higher dimensions are the same as in $n=3$ except some obvious modification.
For the positive mass type theorem in general dimensions, Kwong gave a detail proof of it in \cite{Kwong12}.

In this note, we will prove that some assumptions in Theorem \ref{origin} can be weaken while we can get the same conclusion. This paper is motivated by the paper of Eichmair, Miao and Wang \cite{emw12}, where the authors extended the theorem of Shi and Tam \cite{ShiTam02} to some weaker assumptions. The method we used is similar to that used in \cite{emw12}.

Our main theorem is as follows:

\begin{thm}\label{main}
The conclusion of Theorem \ref{origin} remains valid if the hypothesis that the sectional curvature $K$ of $\Sigma$ is larger than $-k^{2}$ and $\Sigma$ can be isometrically embedded into $\mathbb{H}^{n}_{-k^{2}}$, is weaken to the requirement that $\Sigma$ has scalar curvature $R>-(n-1)(n-2)k^{2}$ and can be isometrically embedded into $\mathbb{H}^{n}_{-k^{2}}$ star-shapedly.
\end{thm}

\vskip 0.2cm \noindent

\begin{rem}
Notice that in $\mathbb{H}^{n}_{-k^{2}}$ we have the Gauss equation
$$K_{ij}=-k^{2}+h_{ii}h_{jj}-h_{ij}^{2},$$
where $A_{0}=(h_{ij})$ denote the second fundamental form of the embedding.
Hence the condition that the sectional curvature of $\Sigma$ is larger than $-k^{2}$, is equivalent to the second fundamental form $A_{0}$ of the embedding is convex.
Also by the Gauss equation, $R>-(n-1)(n-2)k^{2}$ is equivalent to $H_{0}^{2}-|A_{0}|^{2}>0$.
It is well known that this implies the embedding is mean positive (See, e.g. \cite{bc97}).
We denote the cone $\Gamma_{2}\subset{\mathbb{R}}^{n-1}$ to be
\[\Gamma_{2}=\{(\lambda_{1},\ldots,\lambda_{n-1})|\sum_{i=1}^{n-1}\lambda_{i}>0,\sum_{1\leq{i}<j\leq{n-1}}\lambda_{i}\lambda_{j}>0\}.\]
Then, in Theorem \ref{main}, our new assumption is equivalent to that $\Sigma$ can be isometrically embedded into $\mathbb{H}^{n}_{-k^{2}}$ star-shapedly and the principal curvatures satisfy $(\kappa_{1},\ldots,\kappa_{n-1})\in\Gamma_{2}$ everywhere on $\Sigma$.
\end{rem}

\vskip 0.2cm \noindent

\section{Proof of theorem \ref{main}}

Suppose $\iota: \Sigma\rightarrow\mathbb{H}^{n}_{-k^{2}}$ is the star-shaped embedding with $(\kappa_{1},\ldots,\kappa_{n-1})\in\Gamma_{2}$. Following \cite{emw12}, we call
a smooth flow $F: \Sigma\times[0,T]\rightarrow\mathbb{H}^{n}_{-k^{2}}$ is expanding,
if $F$ satisfies the following:
\begin{description}
  \item[1)] $\frac{\partial}{\partial t}F(x,t)=\eta\nu$ with $F(\Sigma,0)$ is $\iota(\Sigma)$, where $\nu$ is the outward pointed unit normal vector and $\eta$ is a positive function on $\Sigma\times[0,T]$,
  \item[2)] on every slice $F(\Sigma,t)$, the principal curvatures satisfy $(\kappa_{1}(t),\ldots,\kappa_{n-1}(t))\in\Gamma_{2}$,
  \item[3)] $F(\Sigma,T)$ is a strictly convex hypersurface in $\mathbb{H}^{n}_{-k^{2}}$.
\end{description}
One example of expanding flow is of the form
\begin{align}\label{inverseflow}
  \frac{\partial}{\partial t}F=\frac{n-2}{n-1}\frac{H}{H^{2}-|A|^{2}}\nu,
\end{align}
whose property is investigated by Gerhardt \cite{Ger11}:

\begin{thm}\label{Ger} {\rm(Gerhardt \cite{Ger11})}
If the initial hypersurface $\Sigma_{0}=\iota(\Sigma)$ is star-shaped and the principal curvatures satisfy $(\kappa_{1}(0),\ldots,\kappa_{n-1}(0))\in\Gamma_{2}$, then the solution for the flow (\ref{inverseflow})
exists for all time $t>0$.
The hypersurfaces $F(\Sigma,t)$ converge to infinity, and remain star-shaped with $(\kappa_{1}(t),\ldots,\kappa_{n-1}(t))\in\Gamma_{2}$. Moreover, the hypersurfaces become strictly convex exponentially fast and more and more totally umbilical in the sense of
\[|h_{ij}-\delta_{ij}|\leq Ce^{-\frac{t}{n-1}},   t>0.\]
\end{thm}

\vskip 0.2cm \noindent

Suppose we have an expanding flow $F: \Sigma\times[0,T]\rightarrow\mathbb{H}^{n}_{-k^{2}}, \frac{\partial}{\partial t}F(x,t)=\eta\nu$.
The pullback of the hyperbolic metric by $F$ to $\Sigma\times[0,T]$ has the form
$g_{\eta}=\eta^{2}dt^{2}+g_{t}$, where $g_{t}$ is the induced metric on the slice $F(\Sigma,t)$. Denote the intrinsic scalar curvature of $g_{t}$ by $R^{t}$,
then $R^{t}>-(n-1)(n-2)k^{2}$. For a positive function $u$, denote $A_{u}$ and $H_{u}$ the second fundamental form and the mean curvature of the slice $\Sigma\times\{t\}$
in the metric $g_{u}=u^{2}dt^{2}+g_{t}$ respectively. It's easy to prove $uA_{u}=\eta A_{\eta}, uH_{u}=\eta H_{\eta}$.
If we denote $\mathbf{1}$ the function that is identically one, then $H_{\mathbf{1}}=\eta H_{\eta}>0$, and $A_{\mathbf{1}}=\eta A_{\eta}$ is bounded.

Suppose we have a positive function $u$ on $\Sigma\times[0,T]$ such that the metric
$g_{u}$ has scalar curvature $R(g_{u})\equiv -n(n-1)k^{2}$.
 By the Gauss equation and the variational formula of mean curvature,
\begin{align}\label{equa1}
\frac{\partial}{\partial t}H_{u}&=-\Delta u-u[|A_{u}|^{2}+Ric(g_{u})(\nu_{u},\nu_{u})]\nonumber\\
                                &=-\Delta u-\frac{u}{2}[R(g_{u})-R^{t}+H_{u}^{2}+|A_{u}|^{2}]\nonumber\\
                                &=-\Delta u+\frac{u}{2}R^{t}+\frac{u}{2}\cdot n(n-1)k^{2}-\frac{1}{2u}[H_{\mathbf{\mathbf{1}}}^{2}+|A_{\mathbf{1}}|^{2}],
\end{align}
where $\Delta$ and $\mid\cdot\mid$ are taken with respect to $g_{t}$. On the other hand,
\[\frac{\partial}{\partial t}H_{u}=\frac{\partial}{\partial t}(\frac{H_{\mathbf{1}}}{u}),\]
then we have
\begin{align}\label{equa}
\frac{\partial}{\partial t}u=\frac{u^{2}}{H_{\mathbf{1}}}\Delta u+\frac{u}{2H_{\mathbf{1}}}[H_{\mathbf{1}}^{2}+|A_{\mathbf{1}}|^{2}]-\frac{u^{3}}{2H_{\mathbf{1}}}[R^{t}+n(n-1)k^{2}]+\frac{u}{H_{\mathbf{1}}}\frac{\partial}{\partial t}H_{\mathbf{1}}.
\end{align}
\begin{prop}\label{solution}
Given any positive function $u_{0}$ on $\Sigma$, there is a positive solution $u$ on $\Sigma\times[0,T]$ of equation (\ref{equa}) such that $u\mid_{t=0}=u_{0}$.
\end{prop}

\vskip 0.2cm \noindent

\begin{proof}
The short time existence of the solution is guaranteed by the positivity of $u_{0}$ and $H_{\mathbf{1}}$. If we can prove that $u$ remains bounded from above and from below by some positive constants, then by standard parabolic theory we obtain the solution on $[0,T]$.
For the upper bound, let $C$ be a constant such that
\[C>\underset{\Sigma\times\{0\}}{\max}u_{0}, \text{\qquad and \qquad} C^{2}>\underset{\Sigma\times [0,T]}{\max}\frac{H_{\mathbf{1}}^{2}+|A_{\mathbf{1}}|^{2}+2\frac{\partial}{\partial t}H_{\mathbf{1}}}{R^{t}+n(n-1)k^{2}}.\]
If $u\geq C$ somewhere on $\Sigma\times[0,T]$, since $u(\Sigma,0)<C$, then there exists $(\tilde{x},\tilde{t})\in\Sigma\times(0,T]$ such that $u(\tilde{x},\tilde{t})=C$ and $u(x,t)\leq C$
for all $t\leq\tilde{t}$. Then since $H_{\mathbf{1}}>0, R^{\tilde{t}}>-(n-1)(n-2)k^{2}$, we have
\[\frac{\partial}{\partial t}u-\frac{u^{2}}{H_{\mathbf{1}}}\Delta u\geq0, \text{\qquad and \qquad}\frac{u}{2H_{\mathbf{1}}}[H_{\mathbf{1}}^{2}+|A_{\mathbf{1}}|^{2}]-\frac{u^{3}}{2H_{\mathbf{1}}}[R^{\tilde{t}}+n(n-1)k^{2}]+\frac{u}{H_{\mathbf{1}}}\frac{\partial}{\partial t}H_{\mathbf{1}}<0, \]
at $(\tilde{x},\tilde{t})$, contradicting equation (\ref{equa}). Thus $u<C$ on $\Sigma\times[0,T]$.

To derive a positive lower bound for $u$, define $\underline{u}=\beta e^{-\gamma t}$, where $\beta<\underset{\Sigma\times\{0\}}{\min}u_{0}$ is a positive constant,
and $\gamma$ is a positive constant such that
\[\gamma>\underset{\Sigma\times [0,T]}{\max}\left\{\beta^{2}[\frac{R^{t}+n(n-1)k^{2}}{2H_{\mathbf{1}}}]-\frac{H_{\mathbf{1}}^{2}+|A_{\mathbf{1}}|^{2}+2\frac{\partial}{\partial t}H_{\mathbf{1}}}{2H_{\mathbf{1}}}\right\}.\]

Let $v=u-\underline{u}$, then $v(\Sigma,0)>0$. If $v\leq0$ somewhere on $\Sigma\times[0,T]$, then there exists $(\tilde{x},\tilde{t})\in\Sigma\times(0,T]$ such that $v(\tilde{x},\tilde{t})=0$ and $v(x,t)\geq 0$
for all $t\leq\tilde{t}$. Then at $(\tilde{x},\tilde{t})$,
\[\frac{\partial}{\partial t}v-\frac{u^{2}}{H_{\mathbf{1}}}\Delta v\leq0,\]
However, by equation (\ref{equa}),
\begin{align}
\frac{\partial}{\partial t}v-\frac{u^{2}}{H_{\mathbf{1}}}\Delta v&=\frac{u}{2H_{\mathbf{1}}}[H_{\mathbf{1}}^{2}+|A_{\mathbf{1}}|^{2}+2\frac{\partial}{\partial t}H_{\mathbf{1}}]-\frac{u^{3}}{2H_{\mathbf{1}}}[R^{\tilde{t}}+n(n-1)k^{2}]+\beta\gamma e^{-\gamma \tilde{t}}\nonumber\\
&=\frac{u}{2H_{\mathbf{1}}}\left\{H_{\mathbf{1}}^{2}+|A_{\mathbf{1}}|^{2}+2\frac{\partial}{\partial t}H_{\mathbf{1}}-\beta^{2}e^{-2\gamma\tilde{t}}[R^{\tilde{t}}+n(n-1)k^{2}]+2H_{\mathbf{1}}\gamma\right\}>0.\nonumber
\end{align}
We get a contradiction. Therefore, $v>0$ on $\Sigma\times[0,T]$, i.e. $u$ has a positive lower bound.
\end{proof}

In order to prove the main theorem, we need the following monotonicity formula:

\begin{prop}\label{monoto}

Suppose $u$ and $\eta$ are two smooth positive functions on $\Sigma\times[0,T]$ such that
$R(g_{u})=-n(n-1)k^{2}$ and $Ric(g_{\eta})=-(n-1)k^{2}$. And $\mathbf{W}: \Sigma\times[0,T]\rightarrow\mathbb{R}^{n,1}$ satisfies the equation
\begin{align}\label{equaW}
  \left\{
     \begin{array}{ll}
       -\frac{H_{\mathbf{1}}}{u\eta}\frac{\partial \mathbf{W}}{\partial t}=\Delta \mathbf{W}-(n-1)k^{2}\mathbf{W},&t\in[0,T]\\
        \mathbf{W}(x,T)=\mathbf{W_{T}}(x)
     \end{array}
   \right.
\end{align}
Here $\mathbf{W_{T}}: \Sigma\rightarrow\mathbb{R}^{n,1}$ is a past-directed non-space-like vector-valued function.
Then for every future-directed null vector $\mathbf{\zeta}\in\mathbb{R}^{n,1}$,
\[\int_{\Sigma\times\{t\}}(H_{\eta}-H_{u})\mathbf{W}\cdot\mathbf{\zeta}d\sigma_{t}\]
is non-increasing in $t$. Here $d\sigma_{t}$ denotes the volume form of $g_{t}$ on $\Sigma\times\{t\}$, and $\cdot$ denotes the Lorentz inner product in $\mathbb{R}^{n,1}$.
\end{prop}

\vskip 0.2cm \noindent

\begin{proof}
First we note that since $\mathbf{W_{T}}$ is past-directed non-space-like, by the standard maximum principle of (backward) parabolic equation, the solution of equation (\ref{equaW}) remains past-directed non-space-like. For simplicity, we denote $W=\mathbf{W}\cdot\mathbf{\zeta}$, then $W\geq0$.

Since $Ric(g_{\eta})=-(n-1)k^{2}$, by the Gauss equation, we have
\[|A_{\eta}|^{2}-H_{\eta}^{2}=-(n-1)(n-2)k^{2}-R^{t}<0.\]
By equation (\ref{equa1}) and (\ref{equaW}), and the variational formula of volume form and mean curvature, we have
\begin{align}
 &\frac{d}{dt}\left(\int_{\Sigma\times\{t\}}(H_{\eta}-H_{u})Wd\sigma_{t}\right)\nonumber\\
=&\int_{\Sigma\times\{t\}}[-\Delta \eta-\eta(|A_{\eta}|^{2}+Ric(g_{\eta}))+\Delta u+\frac{u}{2}(-R^{t}-n(n-1)k^{2})+\frac{1}{2u}(H_{\mathbf{1}}^{2}+|A_{\mathbf{1}}|^{2})]Wd\sigma_{t}\nonumber\\
 &+\int_{\Sigma\times\{t\}}[(H_{\eta}-H_{u})\frac{\partial}{\partial t}W+(\eta^{-1}-u^{-1})H_{\mathbf{1}}^{2}W]d\sigma_{t}\nonumber\\
=&\int_{\Sigma\times\{t\}}(\Delta u-\Delta\eta)Wd\sigma_{t}+\int_{\Sigma\times\{t\}}[H_{\mathbf{1}}(\eta^{-1}-u^{-1})\frac{\partial}{\partial t}W+(\eta^{-1}-u^{-1})H_{\mathbf{1}}^{2}W]d\sigma_{t}\nonumber\\
&+\int_{\Sigma\times\{t\}}[(n-1)k^{2}\eta-\frac{|A_{\mathbf{1}}|^{2}}{\eta}+\frac{1}{2u}(H_{\mathbf{1}}^{2}+|A_{\mathbf{1}}|^{2})+\frac{u}{2}(\eta^{-2}(|A_{\mathbf{1}}|^{2}-H_{\mathbf{1}}^{2})-2(n-1)k^{2})]Wd\sigma_{t}\nonumber\\
=&\int_{\Sigma\times\{t\}}(u-\eta)[\frac{H_{\mathbf{1}}}{u\eta}\frac{\partial}{\partial t}W+\Delta W-(n-1)k^{2}W]d\sigma_{t}+\int_{\Sigma\times\{t\}}(|A_{\eta}|^{2}-H_{\eta}^{2})\frac{(u-\eta)^{2}}{2u}Wd\sigma_{t}\leq0.\nonumber
\end{align}
\end{proof}

Now we begin to prove Theorem \ref{main}.

\begin{proof}[Proof of Theorem \ref{main}.] We consider the flow (\ref{inverseflow}), and denote $\eta=\frac{n-2}{n-1}\frac{H}{H^{2}-|A|^{2}}$ for simplicity. By Theorem \ref{Ger} we can find a $T$ large enough such that $F: \Sigma\times[0,T]\rightarrow\mathbb{H}^{n}_{-k^{2}}, \frac{\partial}{\partial t}F(x,t)=\eta\nu$ is expanding. We will denote the domain between $F(\Sigma,0)$ and $F(\Sigma,T)$ in $\mathbb{H}^{n}_{-k^{2}}$ by $\widehat{\Omega}$, and denote the unbounded domain of $\mathbb{H}^{n}_{-k^{2}}\setminus F(\Sigma,T)$ by $\widetilde{\Omega}$. We denote the hyperbolic metric on $\mathbb{H}^{n}_{-k^{2}}$ by $g_{hyp}$.

Notice that on $\widehat{\Omega}$, $g_{hyp}$ has the form $g_{\eta}=\eta^{2}dt^{2}+g_{t}$, where $g_{t}$ are the induced metrics on $F(\Sigma,t)$. Let $H$ be the mean curvature of $\Sigma$ in $(\Omega,g)$, $H_{\eta}(x,t)$ be the mean curvature of the embedding $F(\cdot,t): \Sigma\times\{t\}\rightarrow\mathbb{H}^{n}_{-k^{2}}$. Taking $u_{0}(x)=\eta(x,0)\frac{H_{\eta}(x,0)}{H(x)}$, and solving the equation (\ref{equa}), we obtain a metric $g_{u}=u^{2}dt^{2}+g_{t}$ such that $R(g_{u})\equiv -n(n-1)k^{2}$ and $H_{u}(x,0)=H(x)$, where $H_{u}(x,t)$ denote the mean curvature of $F(\Sigma,t)$ with respect to the metric $g_{u}$.

On the other hand, since $F(\Sigma,T)$ is strictly convex in $\mathbb{H}^{n}_{-k^{2}}$, we can rewrite the hyperbolic metric $g_{hyp}$ on $\widetilde{\Omega}$ as $g'=d\rho^{2}+\widetilde{g}_{\rho}$ as in \cite{wangyau07}, where $\widetilde{\Sigma}_{\rho}$ denote the level set of the distance function $\rho$ from $\widetilde{\Sigma}_{0}=F(\Sigma,T)$, $\widetilde{g}_{\rho}$ denote the induced metric on $\widetilde{\Sigma}_{\rho}$, $\widetilde{R}^{\rho}$ the intrinsic scalar curvature of $\widetilde{\Sigma}_{\rho}$, $\widetilde{H}_{\rho}$ the mean curvature with respect to $g_{hyp}$. By solving the equation ((2.10) in \cite{wangyau07}):
\begin{align}\label{equa2}
  \left\{
     \begin{array}{ll}
       2\widetilde{H}_{\rho}\frac{\partial v}{\partial\rho}=2v^{2}\Delta v+(v-v^{3})(\widetilde{R}^{\rho}+n(n-1)k^{2}),&\rho\in[0,\infty)\\
        v(p,0)=\frac{\widetilde{H}_{\rho}(p,0)}{H_{u}(p,T)},
     \end{array}
   \right.
\end{align}
we obtain on $\widetilde{\Omega}$ a metric $g''=v^{2}d\rho^{2}+\widetilde{g}_{\rho}$, which is asymptotically hyperbolic (in the sense of \cite{ad98}), with constant scalar curvature $R(g'')\equiv-n(n-1)k^{2}$, and $\widetilde{H}_{v}(p,0)=H_{u}(p,T)$ on $\widetilde{\Sigma}_{0}$, where $\widetilde{H}_{v}(p,\rho)$ is the mean curvature of $\widetilde{\Sigma}_{\rho}$ with respect to the new metric $g''$.

Gluing $(\Omega,g), (\widehat{\Omega},g_{u}), (\widetilde{\Omega},g'')$ along their boundaries, we obtain a spin manifold $\overline{M}$ with Lipschitz asymptotically hyperbolic metric $\overline{g}$ such that $\overline{g}$ is smooth except at $F(\Sigma,0)$ and $F(\Sigma,T)$, the scalar curvature $R(\overline{g})\geq-n(n-1)k^{2}$, and $R(\overline{g})\equiv-n(n-1)k^{2}$ on the unbounded component $\widetilde{\Omega}$.

Now we can use the spinor method to prove the following positive mass type theorem:

\begin{thm}\label{pmm}
Suppose $\widetilde{\Sigma}_{\rho}$, $\widetilde{H}_{\rho}$, $\widetilde{H}_{v}$ are the same as above, and $\mathbf{X}$ denotes the position vector of $\mathbb{H}^{n}_{-k^{2}}$ in $\mathbb{R}^{n,1}$, then for every future-directed null vector $\mathbf{\zeta}\in\mathbb{R}^{n,1}$,
\begin{align}\label{limitmass}
\underset{\rho\rightarrow\infty}{\lim}\int_{\widetilde{\Sigma}_{\rho}}(\widetilde{H}_{\rho}-\widetilde{H}_{v})\mathbf{X}\cdot\zeta d\widetilde{\sigma}_{\rho}\leq0.
\end{align}
\end{thm}

\begin{proof}[Proof of Theorem \ref{pmm}.]
A similar positive mass type theorem when $\Sigma$ can be isometrically embedded into $\mathbb{H}^{n}_{-k^{2}}$ convexly was proved in \cite{wangyau07} for $n=3$, and was proved in \cite{Kwong12} for general dimensions. The only difference in our case is that the metric $\overline{g}$ on the manifold $\overline{M}$ has two corners, near which the metric is only Lipschitz.
Since we are only interested in the asymptotic behavior, the argument in \cite{wangyau07} and \cite{Kwong12} can be carried through with minor modifications. We just describe the general idea of the proof in the following.

A spinor $\psi\in{S}(\mathbb{H}^{n}_{-k^{2}},g_{hyp})$ is said to be a Killing spinor with respect to $\nabla$ if $${\nabla}_{V}\psi+\frac{\sqrt{-1}}{2}kc(V)\psi=0$$
for every $V$, where $\nabla$ and $c$ denote the Riemannian spin connection and the Clifford multiplication with respect to $g_{hyp}$, respectively.

The first important ingredient is that for any future-directed null vector $\mathbf{\zeta}\in\mathbb{R}^{n,1}$, there is a Killing spinor $\phi$ on $(\mathbb{H}^{n}_{-k^{2}},g_{hyp})$ such that
$$|\phi|^{2}_{g_{hyp}}=-2k\mathbf{X}\cdot\zeta.$$
It was proved in Propositions 2.1 and 2.2 of \cite{Kwong12}.

We define the Killing spin connection $\widehat{\nabla}$, Killing Dirac operator $\widehat{D}$ on $(\overline{M},\overline{g})$ by
$$\widehat{\nabla}_{V}\psi={\overline{\nabla}}_{V}\psi+\frac{\sqrt{-1}}{2}k\overline{c}(V)\psi,$$
$$\widehat{D}\psi=\sum_{i=1}^{n}{\overline{c}}(e_{i})\widehat{\nabla}_{e_{i}}\psi,$$
where $\{e_{i}\}_{1\leq{i}\leq{n}}$ is a local orthonormal frame with respect to $\overline{g}$, $\overline{\nabla}$ and $\overline{c}$ denote the Riemannian spin connection and the Clifford multiplication associated to $\overline{g}$, respectively. By the Lichnerowicz type formula (see e.g. \cite{ad98}), we have the following:
\begin{align}
\int_{\overline{M}_{\rho}}\biggl(|\widehat{\nabla}\psi|_{\overline{g}}^{2}+\frac{1}{4}(R(\overline{g})+n(n-1))|\psi|_{\overline{g}}^{2}-|\widehat{D}\psi|_{\overline{g}}^{2}\biggr){d}V
=\int_{\widetilde{\Sigma}_{\rho}}\langle\psi,(\widehat{\nabla}_{\nu_{\rho}}+\overline{c}(\nu_{\rho})\widehat{D})\psi\rangle{d}\widetilde{\sigma}_{\rho},\nonumber
\end{align}
where $\nu_{\rho}$ denotes the outward pointed unit normal vector of $\widetilde{\Sigma}_{\rho}$ in $(\overline{M},\overline{g})$, $\overline{M}_{\rho}$ denotes the bounded domain of $\overline{M}\setminus\widetilde{\Sigma}_{\rho}$.

The second part is to calculate the boundary term of the Lichnerowicz formula. On $\widetilde{\Omega}$, the metrics $g_{hyp}$ and $\overline{g}$ (coincides with $g''$) induce the same metric on the hypersurfaces $\widetilde{\Sigma}_{\rho}$, and the induced metric at the normal direction is related by the function $v$ defined in (\ref{equa2}). Let $A:(T\widetilde{\Omega},g_{hyp})\rightarrow(T\widetilde{\Omega},\overline{g})$ be the Gauge transformation defined by
$A\frac{\partial}{\partial{\rho}}=\frac{1}{v}\frac{\partial}{\partial{\rho}}$ and $AV=V$ for any vector $V$ tangential to $\widetilde{\Sigma}_{\rho}$.
$A$ can be lifted to a map between the associated spinor bundles, $A:S(\widetilde{\Omega},g_{hyp})\rightarrow{S}(\widetilde{\Omega},\overline{g})$, which is an isometry and satisfies
$A(c(V)\psi)=\overline{c}(A(V))A(\psi)$ for $\psi\in{S}(\widetilde{\Omega},g_{hyp})$ and $V\in(T\widetilde{\Omega},g_{hyp})$. We denote $\phi''=A\phi$ (on $\widetilde{\Omega}$), then
\begin{align}\label{e1}
|\phi''|_{\overline{g}}^{2}=|\phi|_{g_{hyp}}^{2}=-2k\mathbf{X}\cdot\zeta.
\end{align}
Since we are only interested in the asymptotic behavior, we can extend $\phi''$ smoothly on the whole $\overline{M}$ by multiplying a cut-off function and we still denote it by $\phi''$.

By considering the hypersurface Dirac operator, one can prove the following (see Propositions 2.3 and 2.4 of \cite{Kwong12}):
\begin{align}\label{e2}
\int_{\widetilde{\Sigma}_{\rho}}\langle\phi'',(\widehat{\nabla}_{\nu_{\rho}}+\overline{c}(\nu_{\rho})\widehat{D})\phi''\rangle{d}\widetilde{\sigma}_{\rho}
=\frac{1}{2}\int_{\widetilde{\Sigma}_{\rho}}(\widetilde{H}_{\rho}-\widetilde{H}_{v})|\phi''|_{\overline{g}}^{2}{d}\widetilde{\sigma}_{\rho}.
\end{align}

The third important ingredient is to prove that there exists a spinor $\bar{\phi}$ on $(\overline{M},\overline{g})$ with $\widehat{D}\bar{\phi}=0$ and the asymptotic behavior
\begin{align}\label{e3}
\underset{\rho\rightarrow\infty}{\lim}\int_{\widetilde{\Sigma}_{\rho}}\langle\phi'',(\widehat{\nabla}_{\nu_{\rho}}+\overline{c}(\nu_{\rho})\widehat{D})\phi''\rangle{d}\widetilde{\sigma}_{\rho}
=\underset{\rho\rightarrow\infty}{\lim}\int_{\widetilde{\Sigma}_{\rho}}\langle\bar{\phi},(\widehat{\nabla}_{\nu_{\rho}}+\overline{c}(\nu_{\rho})\widehat{D})\bar{\phi}\rangle{d}\widetilde{\sigma}_{\rho}.
\end{align}
The existence of $\bar{\phi}$ is proved mainly by Lax-Milgram Theorem and a regularity argument. When applying Lax-Milgram Theorem, only the asymptotic behavior of $\overline{g}$, the boundedness of $R(\overline{g})$, and the condition $R(\overline{g})\geq-n(n-1)k^{2}$ are used. Because $\overline{g}$ is smooth except at the two corners and is Lipschitz near these two corners, the regularity argument in \cite{ShiTam02} (see also \cite{liuyau06}) can go through.

Finally, by the Lichnerowicz formula, since $\widehat{D}\bar{\phi}=0$,
\begin{align}\label{e4}
\int_{\widetilde{\Sigma}_{\rho}}\langle\bar{\phi},(\widehat{\nabla}_{\nu_{\rho}}+\overline{c}(\nu_{\rho})\widehat{D})\bar{\phi}\rangle{d}\widetilde{\sigma}_{\rho}
=\int_{\overline{M}_{\rho}}\biggl(|\widehat{\nabla}\bar{\phi}|_{\overline{g}}^{2}+\frac{1}{4}(R(\overline{g})+n(n-1))|\bar{\phi}|_{\overline{g}}^{2}\biggr){d}V\geq0.
\end{align}

Combining (\ref{e1}), (\ref{e2}), (\ref{e3}), (\ref{e4}), we obtain (\ref{limitmass}).
\end{proof}

Now we solve the equation ((5.3) in \cite{wangyau07})
\[
  \left\{
     \begin{array}{ll}
       \frac{\widetilde{H}_{\rho}}{v}\frac{\partial \mathbf{\widetilde{W}}}{\partial\rho}=-\Delta \mathbf{\widetilde{W}}+(n-1)k^{2}\mathbf{\widetilde{W}},&\rho\in[0,\infty)\\
        \underset{\rho\rightarrow\infty}{\lim}e^{-k\rho}\mathbf{\widetilde{W}}(p,\rho)=-\underset{\rho\rightarrow\infty}{\lim}e^{-k\rho}k\mathbf{X}(p,\rho),
     \end{array}
   \right.
\]
and solve the equation (\ref{equaW}) with $\mathbf{W}(\cdot,T)=\mathbf{\widetilde{W}}(\cdot,0)$, and finally obtain $\mathbf{W^{0}}(x)=-\mathbf{W}(x,0)$. Since $\mathbf{X}$ is future-directed time-like, by the maximum principle of (backward) parabolic equation, it is easy to prove that both $\mathbf{\widetilde{W}}$ and $\mathbf{W}$ are past-directed time-like, hence $\mathbf{W^{0}}$ is future-directed time-like. By the monotonicity formula in \cite{wangyau07} (see Proposition 5.3 of \cite{wangyau07}) and our monotonicity formula in Proposition \ref{monoto}, together with (\ref{limitmass}), for any future-directed null vector $\mathbf{\zeta}\in\mathbb{R}^{n,1}$,
\begin{align}
&\int_{\Sigma}(H_{0}-H)(-\mathbf{W^{0}})\cdot\mathbf{\zeta}d\sigma=\int_{\Sigma\times\{0\}}(H_{\eta}(0)-H_{u}(0))\mathbf{W}\cdot\mathbf{\zeta}d\sigma_{t}\nonumber\\
\geq&\int_{\Sigma\times\{T\}}(H_{\eta}(T)-H_{u}(T))\mathbf{W}\cdot\mathbf{\zeta}d\sigma_{t}=\int_{\widetilde{\Sigma}_{0}}(\widetilde{H}_{0}-\widetilde{H}_{v}(0))\mathbf{\widetilde{W}}\cdot\mathbf{\zeta}d\widetilde{\sigma}_{\rho}\nonumber\\
\geq&\underset{\rho\rightarrow\infty}{\lim}\int_{\widetilde{\Sigma}_{\rho}}(\widetilde{H}_{\rho}-\widetilde{H}_{v}(\rho))\mathbf{\widetilde{W}}\cdot\mathbf{\zeta}d\widetilde{\sigma}_{\rho}\geq0.\nonumber
\end{align}
Therefore,
\[\int_{\Sigma}(H_{0}-H)\mathbf{W^{0}}d\sigma\]
is a future-directed non-space-like vector.
\end{proof}

\noindent\textbf{Acknowledgments}.The author would like to express his gratitude to his advisor Professor Bing-Long Chen for his constant encouragement and careful guidance. The author is grateful to the referees for careful reading and for critical comments to improve this paper.

\end{document}